\documentclass[11pt] {amsart}
\usepackage{amsthm}
\usepackage{amsxtra}
\usepackage{amssymb}
\usepackage{graphicx}

\newtheorem{theorem}{Theorem}

\newtheorem{lemma}{Lemma}

\renewcommand{\Re}{\mathop{\rm Re}\nolimits}
\renewcommand{\Im}{\mathop{\rm Im}\nolimits}

\newcommand{\CIc}{C^\infty_{\rm{c}}}
\newcommand{\RR}{\mathbb R}
\newcommand{\Spec}{\operatorname{Spec}}
\newcommand{\bl}{\begin{flushleft}}
\newcommand{\el}{\end{flushleft}}
\newcommand{\br}{\begin{flushright}}
\newcommand{\ert}{\end{flushright}}
\newcommand{\bc}{\begin{center}}
\newcommand{\ec}{\end{center}}

\newcommand{\isect}{\bigcap}

\newcommand{\recip}[1]{\frac{1}{#1}}

\newcommand{\numList}{\begin{enumerate}}
\newcommand{\enumList}{\end{enumerate}}
\newcommand{\sumLim}[3]{\sum_{#1=#2}^{#3}}

\newcommand{\e}{\epsilon}

\newcommand{\re}{\mathbb{R}}

\newcommand{\nn}{\nonumber\\}
\newcommand{\la}{\langle}
\newcommand{\ra}{\rangle}
\newcommand{\negint}{{\int\negthickspace\negthickspace\negthickspace\negthinspace -}}

\numberwithin{equation}{section}

\usepackage{dsfont}
\usepackage{amsfonts}
\usepackage{textcomp}
\usepackage{lmodern}
\usepackage{pxfonts}

\title [Nonlinear instability in a semiclassical problem]{Nonlinear instability in a semiclassical problem}
\author[J. Galkowski]{Jeffrey Galkowski}
\address{Mathematics Department, University of California, Berkeley, 
CA 94720, USA}
\email{jeffrey.galkowski@math.berkeley.edu}

\begin{document}

\begin{abstract}
We consider a nonlinear evolution problem with an asymptotic
parameter and construct examples in which the linearized
operator has spectrum uniformly bounded away from 
$ \Re z \geq 0 $ (that is, the problem is spectrally stable), yet the nonlinear
evolution blows up in short times for arbitrarily small
initial data.

We interpret the results in terms of semiclassical pseudospectrum
of the linearized operator: despite having the spectrum in 
$ \Re z < -\gamma_0 < 0 $, the resolvent of the linearized operator grows very quickly in parts of 
the region $ \Re z > 0 $. We also illustrate the results numerically.
\end{abstract}

\maketitle

\section{Introduction}
For a large class of nonlinear evolutions the size of the
resolvent has been proposed as an explanation of instability 
for spectrally stable problems. Celebrated examples include the 
plane Couette flow, plane Poiseuille flow and plane flow -- 
see Trefethen-Embree \cite[Chapter 20]{ET} for discussion and
references. Motivated by this we consider the mathematical question 
of evolution involving a small parameter $ h $ (in fluid
dynamics problem we can think of $ h $ as the reciprocal of
the Reynolds number) in which the linearized operator has 
the spectrum lying in $ \Re z <  - \gamma_0 < 0 $, uniformly 
in  $h $,  yet 
the the solutions of the nonlinear equation blow up at time 
$ O ( 1 ) $ for data of size $ O ( \exp ( - c/h ) $.  

We know of one rigorous example of such a phenomenon
given by Sandstede-Scheel \cite{Sand03}. They considered
$ u_t = u_{xx} + u_{x} + u^3$ on $ [ 0, \ell ] $ with Dirichlet 
boundary conditions, and showed that 
blow up occurs with arbitrarily
small initial data as $ \ell \rightarrow \infty $.
In that problem $ h = 1/\ell $.  
The paper \cite{Sand03} is our starting point and we use its maximum
principle approach to obtain results
for suitable operators in any dimension. In addition, we
emphasize the connection with the semiclassical 
pseudospectrum and provide some numerical comparisons.

We consider a semiclassical nonlinear evolution equation 
\begin{equation}
\label{eqn:main}
hu_t=P(x, hD)u+u^3\,, \ \ \ x \in \RR^d \,, \ \ t \geq 0 \,. 
\end{equation}
where $ P ( x, h D ) $ is the following semiclassical differential operator 
\begin{equation}
\label{eqn:L}
P(x, hD):= - \left( (hD)^2 + V(x)\right) +ih\la \nabla \rho ,D\ra +\mu \,, \ \ \  D_{j}
:= \frac 1 i \partial_{j}\quad D^2=-\Delta.
\end{equation}
Here, $V\in C^\infty$ is a potential function with
\begin{equation}
\label{eqn:PotentialProp}
\begin{cases} V(x)\geq C\la x\ra^k\text{ on } |x|\geq M,&
|\partial^\alpha V(x)|\leq C_\alpha \la x\ra^k,\\
V(x_0)= 0\text{ for some }x_0\in\re^d,&V(x)\geq 0,\end{cases}
\end{equation}
 for some $C,C_\alpha,k,M>0$. Also, $\rho\in C^\infty$ has the properties
\begin{equation}
\label{eqn:vfieldProp}
|\partial^2_{ij}\rho|\leq C_{ij}\la x\ra ^k,\quad
|\nabla\rho|^2\geq 4(\mu+\gamma_0),
\end{equation} 
for some $C_{ij},\gamma_0,N>0$.
Finally, $\mu>0$. 

\medskip

\noindent
{\bf{Remark:}} All of our results hold for weaker assumptions on the growth of $V$ and $\rho$, however \eqref{eqn:PotentialProp} and 
\eqref{eqn:vfieldProp} are convenient for our purposes.

We will show in section \ref{sec:spectra} that for $V(x)$ and $\rho$
as in (\ref{eqn:PotentialProp}) and (\ref{eqn:vfieldProp})
respectively, the linearized problem is spectrally stable, that is,
the spectrum is bounded away from $\Re z\geq 0$ uniformly in $ h $.
Yet, we also show that \eqref{eqn:main} 
has an unstable equilibrium at $u\equiv 0$ for all potentials $V(x)$ satisfying (\ref{eqn:PotentialProp}) and all  $\rho$ satisfying (\ref{eqn:vfieldProp}). Specifically, we show 
\begin{theorem}\label{thm:main}
Fix $\mu>0$. Then, for each 
\[ 0 < h < h_0 \,, \]
where $ h_0 $ is small enough, there exists
\[ u_0 \in \CIc ( \RR^n ) \,, \ \ u_0 \geq 0 \,, \ \ \| u_0 \|_{ C^p} \leq
\exp\left(-\recip{Ch}\right) \,, \  \ p = 0 , 1 , \dots \,, \]
such that the solution to \eqref{eqn:main} with $ u ( x, 0 ) =  u_0 (
x ) $, satisfies 
\[   \| u ( x, t ) \|_{L^\infty } \longrightarrow \infty , \ \ t
\longrightarrow T \,, \]
where
\[   T = O (1) \, . \]
\end{theorem}

A nice example for which our assumptions hold is (\ref{eqn:L}) with $x\in \re$, $V(x)=x^2$, and $\la\nabla \rho,D\ra=D_x$. That is
\begin{equation}
\label{eqn:L1d}
P_1(x,hD):=-\left((hD_x)^2+x^2\right)+ihD_x +\mu\,,  \quad x\in \re
\end{equation}
It is easy to see (and will be described in Section \ref{sec:spectra})
that 
\[ \begin{split}  \Spec ( P_1 ( x , h D) ) & = \{ \mu -1/4  - h (2 n + 1) \; : \; n = 0,
1, 2 \cdots \} \\ & \subset \{ z \; : \; \Re z \leq \mu  -1/4 \} \,. 
\end{split}  \]
For $\mu>\recip{4}$ the spectrum intersects the right half plane and thus instability of the linear problem follows. We are interested in the range
$0<\mu<\recip{4}$, where we will 
relate the instability of $u=0$ to the presence of pseudospectrum in the right half plane. 

For more about \eqref{eqn:L1d} see \cite[Chapter 12]{ET}. 
In particular, Cossu-Chomaz \cite{CC} relate it to the linearized
Ginzburg-Landau equation and analyze resolvent of \eqref{eqn:L1d} and
the norm of the semigroup $e^{P_1(x,hD)t}$
numerically.

The operator \eqref{eqn:L1d} is also closely related to the
advection-diffusion operators mentioned above, $ -D_y^2 + i D_y = \partial_y^2
+ \partial_y $, on $ [ 0, \ell ] $, with, say, Dirichlet boundary 
conditions; see \cite[Chapter 12]{ET} for a discussion and
references. When rescaled using $ x = y /\ell $, $ h = 1/l $
the operator becomes the semiclassical operator $ - ( hD)^2 + i h
D_x $ on $ [ 0 , 1 ] $. When the domain is extended to $\re$, the potential $ x^2 $ is added to 
$ (h D_x)^2 $ to produce a confinement similar to a boundary. 

We relate the blow-up of solutions to \eqref{eqn:main} to the presence of pseudospectrum of \eqref{eqn:L} in the right half plane. However, because estimates on semigroups for (\ref{eqn:main}) with quasimode initial data are poor, we are unable to exhibit blow-up starting from a quasimode. Instead, we present a simple and explicit construction of quasimodes for $P(x,hD)$ (for a more general setting see \cite{zworski04}). We then use these quasimodes as initial data in numerical simulations and observe that, although in some cases the ansatz solution blows up more quickly, the solutions with quasimode initial data behave similarly to what is expected from a pure eigenvalue for \eqref{eqn:L} with positive real part.

\begin{figure}[htbp]
\includegraphics[width=6.5in]{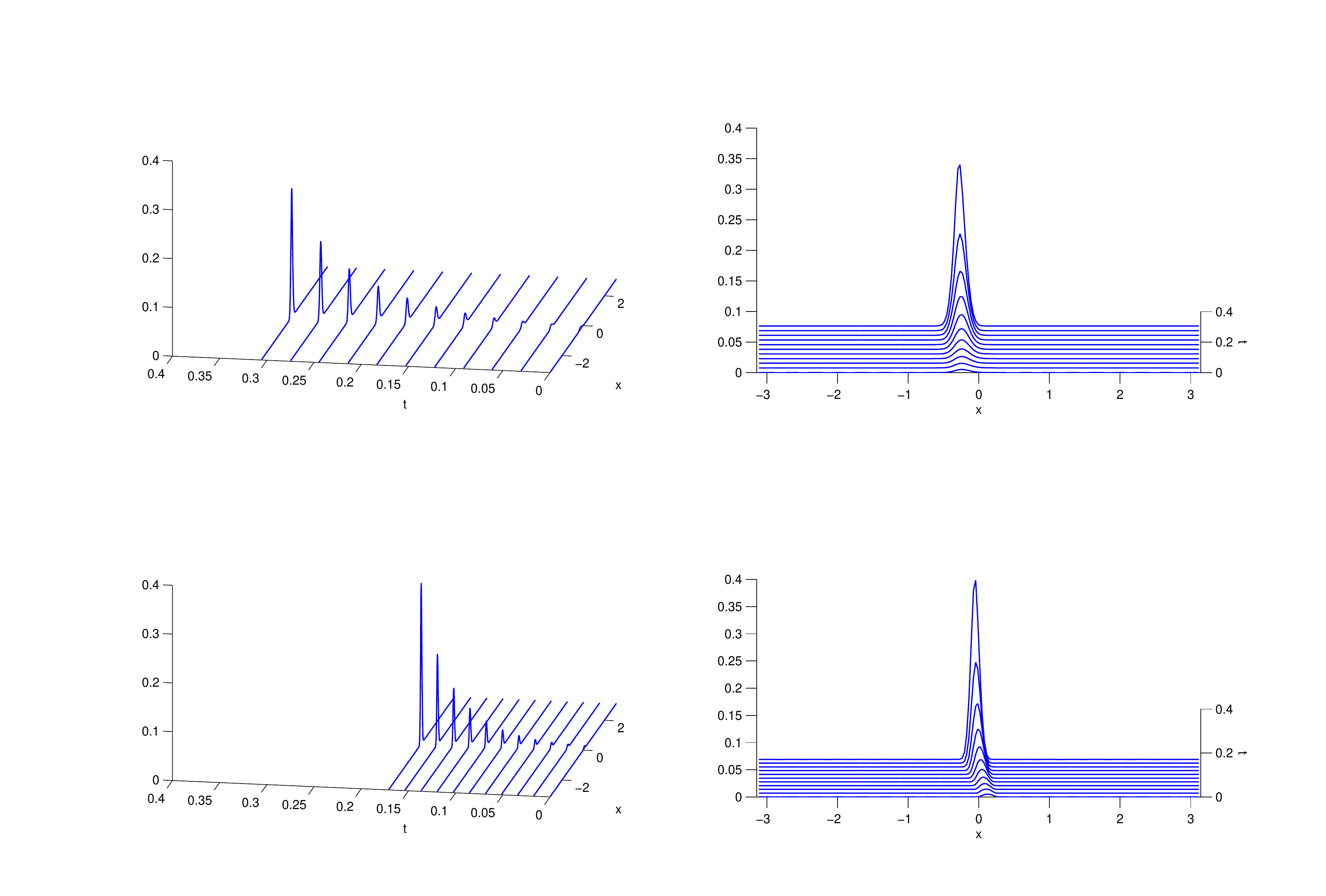}
\begin{center}
\caption{\label{f:2}
The plot shows numerical simulations of the evolution of
(\ref{eqn:main}) with $h=1/193$ and two initial data.  The evolution with initial data a real valued $O(h^3)$ error quasimode with eigenvalue $z=\recip{16}$ is shown in the top two graphs and that with the ansatz constructed in the proof of Lemma \ref{lem:O(1)ndim} as initial data is shown in the bottom two graphs. We observe that, when the initial data is a quasimode, blowup occurs in time $\approx 0.3$, while for ansatz initial data, blow-up occurs in time $\approx 0.175$. However, as would be expected from eigenfunction initial data, we see that the solution with quasimode initial data exhibits little transport to the left. On the other hand, the ansatz transports left significantly.}
\end{center}
\end{figure}

The paper is organized as follows. In Section \ref{sec:spectra} we
review the definitions of spectra and pseudospectra and discuss them
for our class of operators. In Section \ref{sec:odq} we give a
construction of quasimodes for
one dimensional problems. Although the results are known, (see
\cite{Davi98},\cite{zworski04},\cite{zworski01}) a self-contained presentation 
is useful since we need the quasimodes for our numerical experiments.
Also, there is no reference in which analytic potentials (for which
quasimodes have $ O ( \exp ( - c/h) ) $ accuracy) is treated
by elementary methods in one dimension. Section \ref{sec:i} 
is devoted to the proof of Theorem \ref{thm:main} using heat
equation methods. Finally, in Section \ref{sec:numerical} we 
report on some numerical experiments which suggest that
quasimode initial data gives more natural blow-up and that blow-up occurs at complex energies.

\noindent
{\sc Acknowledgemnts.} The author would like to thank Maciej Zworski for suggesting the problem and for valuable discussion, guidance, and advice. Thanks also to Laurent Demanet and Trever Potter for allowimg him to use their MATLAB codes, Justin Holmer for informing him of the paper by Sandstede and Scheel, and Hung Tran for comments on the maximum principle in Lemma \ref{lem:O(1)ndim}. The author is grateful to the National Science Foundation for partial support under grant DMS-0654436 and under the National Science Foundation Graduate Research Fellowship Grant No. DGE 1106400.

\section{Spectrum and Pseudospectrum}
\label{sec:spectra}
We do not use the results of this section to prove Theorem
\ref{thm:main}. Instead, we present them to emphasize the connection of the 
size of the resolvent with instability. We believe that instability
based on quasimodes would be more natural and allow for proof of instability at complex energies. We illustrate this with numerics in Section \ref{sec:numerical}.

To describe the spectrum of $P(x, hD)$, we observe that 
$$e^{\frac{\rho(x)}{2h}}P(x, hD)e^{-\frac{\rho(x)}{2h}}=-\left[(hD)^2+V(x)+\recip{4}
|\nabla \rho|^2+\frac{h}{2}\Delta \rho\right]+\mu.$$
Thus, the spectrum of $P(x, hD)$ is given by that of a Schr\"{o}dinger
operator with potential $V(x)+\recip{4}|\nabla
\rho|^2+\frac{h}{2}\Delta\rho$. Since $V(x)$ and $\rho$ have the properties given in (\ref{eqn:PotentialProp}) and (\ref{eqn:vfieldProp}) respectively, $P(x,hD)$ has a discrete spectrum that has real part bounded above by $-\gamma_0$ (see for instance \cite[Section 6.3]{EZB}).

We now examine the pseudospectral properties of (\ref{eqn:L}).

\noindent
{\bf Definition.} 
Let $Q(x,hD)$ be a second order semiclassical differential operator. Then, $z\in \Lambda(Q)$ if and only if $\exists$ $u(h)\in H^2(\re^d)$ such that $\|u\|_{L^2}=1$ and 
$$\left\|\left(Q(x,hD)-z\right)u(h)\right\|_{L^2}=O(h^\infty).$$
We say $z$ is in the semiclassical \emph{pseudospectrum} of $Q$ if $z\in \overline{\Lambda(Q)}$. 

\noindent
{\bf Remark.} We note that for $ z \in \Lambda ( Q ) $,  $\| Q ( x , h D
) - z ) ^{-1} \| \geq h^{-N} / C_N $, for any $ N$. This relates our
definition to the more standard defintions of pseudospectra in terms of
the resolvent. For discussion and generalizations see Dencker
\cite{De} and Pravda-Starov \cite{KPS}.


The criterion for $ z \in \Lambda ( Q ) $ is  based on H\"ormander's
bracket condition (see Zworski \cite{zworski01} and
Dencker-Sj\"ostrand-Zworski \cite{zworski04}):
\begin{equation}
\label{eqn:inPseudo} Q(x_0,\xi_0)=z\text{ and }\{\Re Q,\Im Q\}(x_0,\xi_0)<0,
\end{equation}
 then  $z\in \Lambda (Q)$. We use this condition to show that the pseudospectrum of $P(x,hD)$ nontrivially interesects the right half plane. Specifically,  
\begin{lemma}
For $P(x,hD)$ given by (\ref{eqn:L}), $\overline{\Lambda(P(x,hD))}\isect \{\Im z =0\}=(-\infty ,\mu]$. 
\end{lemma}
\begin{proof}
First, observe that 
$$P(x,\xi)=-|\xi|^2+i\la \nabla \rho,\xi \ra-V(x)+\mu\quad \text{ and }\{\Re P,\Im P\}=-2\la \partial^2\rho\xi,\xi\ra+\la\nabla V,\nabla \rho\ra.$$
We have assumed $\Im z=0$. Therefore, we need only show that, for a dense subset $U\subset (-\infty, \mu]$, $y\in U$ implies that there exists $x$ such that (\ref{eqn:inPseudo}) holds for the symbol $P(x,\xi)$, at $(x,0)$ with $z=y$. 

We proceed by contradiction. Suppose there is no such $U$. Then, there exists $O\subset [0,\infty)$ open such that for all $x\in V^{-1}(O)$, $\la\nabla V, \nabla \rho\ra(x,0)\geq 0$. Let $\varphi_t:=\exp(ti\la\nabla \rho,D\ra)$ be the integral flow of $i\la\nabla\rho,D\ra$ and $x_0\in \re^d$ have $V(x_0)=0$. Define $f(t):= V(\varphi_t(x_0))$. Then $\partial_tf=\la\nabla V(\varphi_t(x_0)),\nabla \rho(\varphi_t(x_0))\ra$.

Suppose that $\varphi_t(x_0)$ escapes every compact set as $|t|$ increases. Then \eqref{eqn:vfieldProp} implies that $f(t)\to \infty$ as $|t|$ increases. Let $w\in O$ and $t_0:=\inf\left\{t\in \re\text{ : }f(t)=w\right\}$. Then $t_0$ is finite since $w\geq f(0)$ and $f(t)\to \infty$. Together, $f(t_0)=w\in O$ and $f^{-1}(O)$ open imply the existance of $\delta >0$ such that for $t\in (t_0-\delta,t_0+\delta)$, $f(t)\in O$. But, $f(t)\in O$ implies $f'(t)\geq 0$. Therefore, $f(t)\leq w$ for $t\in (t_0-\delta, t_0)$ and thus, since $f(t)\to \infty$, there exists $t<t_0$ such that $f(t)=w$, a contradiction. 

We have shown that there is a dense subset $U\subset (-\infty, \mu]$ with $U\subset \Lambda(P)$. Hence $(-\infty,\mu]\subset \overline{\Lambda (P)}$. Next, observe that $\sup\Re P(x,\xi)=\mu$ and thus, $\overline{\Lambda(P(x,hD))}\isect \{\Im z =0\}=(-\infty ,\mu]$ as desired.

To finish the proof, we need only show that $\varphi_t(x_0)$ escapes every compact set. Suppose the flow at $x_0$ exists for all $t\in \re$. Define $h(t):=\rho(\varphi_t(x_0))$. Then $\partial_t h=|\nabla \rho|^2\geq c>0$ and we have that $h\to\pm\infty $ as $t\to \pm\infty$. But, $\rho\in C^\infty$ and is therefore bounded on every compact set. Thus, $\varphi_t(x_0)$ escapes every compact set as $t\to \pm \infty$.

Now, suppose the flow at $x_0$ is not global. Then, $\varphi_t(x_0)$ is an integral curve of $i\la\nabla \rho,D\ra$ with $t$ domain a proper subset of $\re$. Thus, as proved in \cite[Lemma 17.10]{lee}, $\varphi_t(x_0)$ escapes every compact set.
\end{proof}

Putting this together with our discussion of the spectrum of $P(x,hD)$, we have that for $0<\mu$ and $\rho$ as in \eqref{eqn:vfieldProp}, although $\Spec(P)$ is bounded away from $\Re z\geq 0$, $\overline{\Lambda(P)}$ nontrivially intersects $\Re z\geq 0$. 

\begin{figure}[htbp]
\begin{center}
\includegraphics[width=3in]{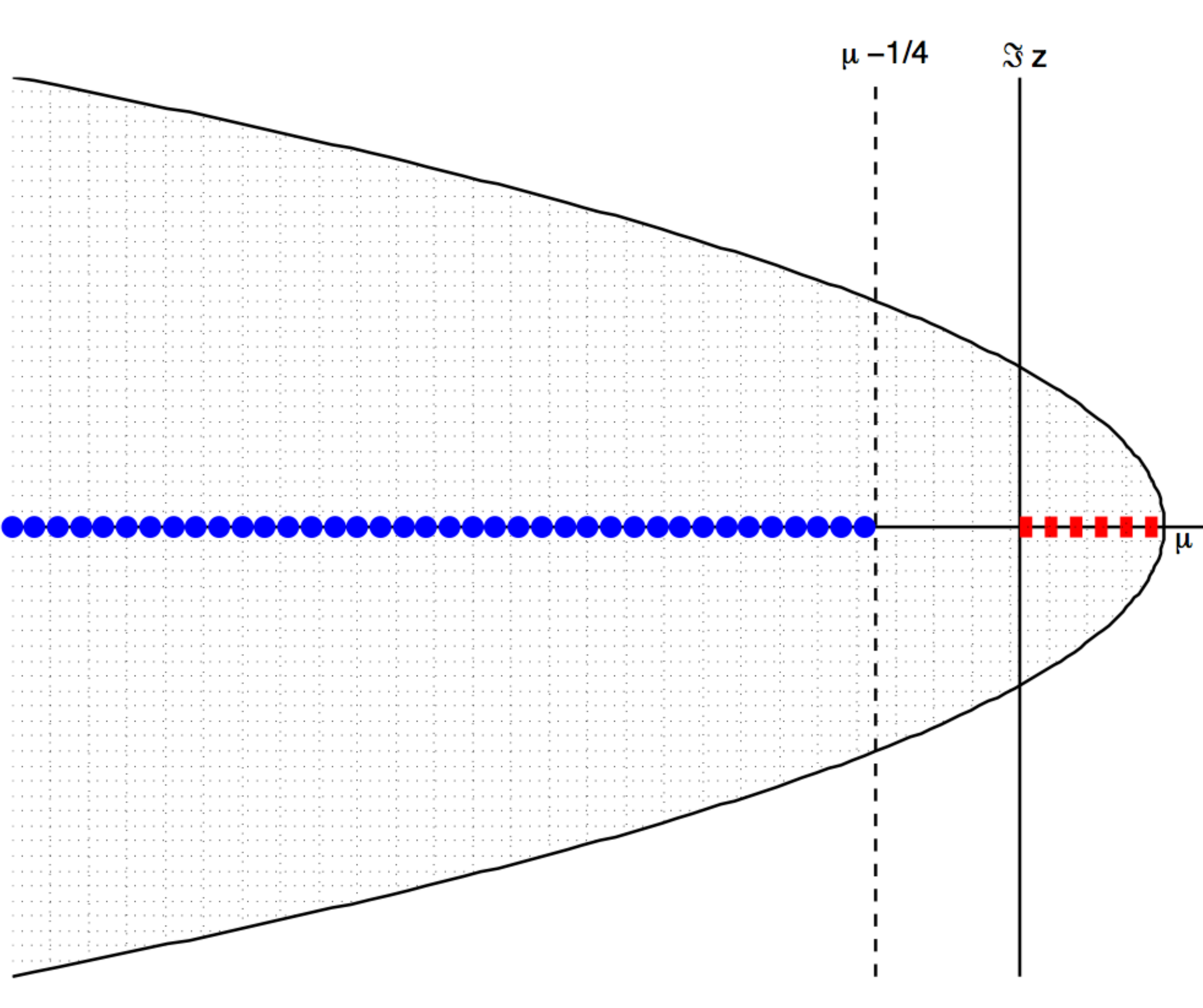}
\caption{\label{f:spectrum}
We see that the spectrum of (\ref{eqn:L1d}) (blue dots) is bounded away from $\Re z=0$, while the pseudospectrum (shaded region) enters the right half plane. The region for which we prove blow-up corresponds to the dashed red line. }
\end{center}
\end{figure}

For the specific case, $V(x)=|x|^2$, and  $\nabla\rho$ constant with $|\nabla \rho|=1$, the above argument gives us that 
$$\text{spec}\left(P(x, hD)\right)=\left\{-(2n+1)h+(\mu-\recip{4}):n\geq 0\right\}.$$
In addition, the pseudospectrum is given by, 
$$\overline{\Lambda\left(P(x, hD)\right)}=\left\{z:\Re z\leq -(\Im z)^2+\mu \right\}.$$

We see that for $\mu>\recip{4}$, the spectrum interesects the right half plane and so instability of $u\equiv 0$ is a classical result. However, for $0<\mu<\recip{4}$, the spectrum is bounded away from the $\Re z\geq 0$ and only the pseudospectrum enters the right half plane. Yet, in the regime $0<\mu<\recip{4}$, we will show that $u\equiv 0$ is unstable and, moreover, for arbitrarily small initial data, the solution blows up in finite time.

\section{One Dimensional Quasimodes}
\label{sec:odq}

We proceed by constructing quasimodes for operators in the one dimensional case with $i\la\nabla\rho,D\ra=\partial_x$. We implement WKB expansion for the quasimode following the method used in \cite{Davi98}. Let 
\begin{equation}
\label{eq:Pxh} P(x, h D) :=-\left(hD_x\right)^2+ihD_x+V \,,
\end{equation}
where $V\in C^\infty$ and $V$ may be complex.\\

\noindent{\bf Remark.} The following theorem is a special case of general theorems about 
quasimodes \cite[Theorems 2 and 2']{zworski04}. For the reader's
convenience we present a direct proof in the spirit of Davies \cite{Davi98}. 
\begin{theorem}
\label{prop:constructQuasi}
Suppose that $ P ( x , h D) $ is given by \eqref{eq:Pxh} and that
\[  z = - \xi_0^2 + i \xi_0  + V ( x_0 ) \,,\]
where $ x_0 $ satisfies the condition that 
\[  \Re V' ( x_0 ) > 0  \,. \]

There exists an $h-$dependent function $\varphi\in C_c^\infty(\re)$, 
such that $\|\varphi\|_{L^2}=1$ and
$$\|(P(x, h D) -z)\varphi\|_{L^2} = O (
h^\infty ) . $$

In addition $ \varphi $ is microlocalized to $ ( x_0, \xi_0) $ in the 
sense that for every $ g \in \CIc ( \RR^n \times \RR^n ) $ vanishing
in a neighbourhood of $ ( x_0, \xi_0 ) $,
\[ \|   g ( x, h D ) \varphi \|_{L^2}= O ( h^\infty ) \,. \]

When $ V $ is real analytic than we can find $\varphi $ such that
$$\|(P(x, h D) -z)\varphi\|_{L^2}\leq C \exp (
- 1 / C h ) \,. $$

\end{theorem}


\begin{proof}
Let $\chi\in C_c^\infty(\re)$ with $\chi(x)=1$ if $|x|<\delta/2$ and $\chi(x)=0$ if $|x|>\delta$ where $\delta $ will be determined below. Define $f:=\exp(i\psi/h)a(x)$ where 
$$a(x)=\sum_{m=0}^{N-2}a_m(x)h^m.$$ 
Finally, let $g(x_0+x):=\chi(x)f(x)$ for all $x\in \re$. 

We will find appropriate $a_m$ and $\psi$ in what follows. First, by a simple computation
$$\left(P(x, h D)-z \right)f=\left(\sumLim{m}{0}{N}h^m\phi_m\right)e^{i\psi/h}$$
where $\phi_m$ are inductively defined by
\begin{eqnarray}
\phi_m&:=&(-(\psi ')^2+i\psi '+V-z)a_m+i\psi '' a_{m-1}+(2i\psi '+1)a_{m-1}'+a_{m-2}'',\nonumber
\end{eqnarray}
where we use the convention that $a_m\equiv 0$ for $m>N-2$ or $m<0$. Now, we set $\phi_m=0$ for $0\leq m\leq N-1$. Given that $\delta$ is small enough, this will enable us to determine all $a_m$ as well as $\psi$.

Observe that, using the condition, $\phi_0=0$, we obtain
$$\psi '^2-i\psi '=V-z.$$
Now, letting $z=-\xi_0^2+i\xi_0+V(x_0)$, we have a complex eikonal equation
$$\left(\psi '-\frac{i}{2}\right)^2=V(x_0+x)-V(x_0)+\left(\xi_0-\frac{i}{2}\right)^2.$$
Then, letting $\tilde{\psi}=\psi-\frac{i}{2}x$, we have
\begin{eqnarray}
\tilde{\psi}(x)&:=&\int_{0}^x\left(V(x_0+t)-V(x_0)+\left(\xi_0-\frac{i}{2}\right)^2\right)^{1/2}dt\nn 
&=&\int_0^x\left(\xi_0-\frac{i}{2}\right)\left(1+\left(\xi_0-\frac{i}{2}\right)^{-2}\left(V(x_0+t)-V(x_0)\right)\right)^{1/2}dt\nn 
&=&\left(\xi_0-\frac{i}{2}\right)\left(x+\recip{4}\left(\xi_0-\frac{i}{2}\right)^{-2}V'(x_0)x^2+O(x^3)\right)\nonumber
\end{eqnarray}
and hence
$$\psi=\xi_0x+\frac{\left(\xi_0+\frac{i}{2}\right)}{4\left(\xi_0^2+\frac{1}{4}\right)}V'(x_0)x^2+O(x^3).$$

Now, we have assumed that $\Re V'(x_0)>0$. Therefore there exists $\gamma>0$ such that
$$\gamma x^2\leq\Im\psi(x)\leq 3\gamma x^2$$
for all small enough $x$ and $h$. Also, for $x$ and $h$ small enough 
$$\theta:=(2i\psi '+1)^{-1}$$ 
satisifies $|\theta(x)|\leq \beta$. We choose $\delta>0$ small enough so that these conditions both hold for $0<h<\delta^2$ and $|x|<\delta$. 

The condition $\phi_{m+1}=0$, implies
\begin{eqnarray}
a_m'&=&-\theta(i\psi ''a_m+a_{m-1}'')\nonumber
\end{eqnarray}
with the convention that $a_{-1}\equiv 0$ and initial conditions, 
$$a_0(0)=1,\quad a_m(0)=0,\text{ }m>0.$$
Putting $G(x):=\int_0^xi\psi '' (y)\theta(y)dy$ we obtain $a_0=\exp(-G(x))$ and 
\begin{equation}
\label{eqn:inductA}
a_{m+1}(x):=-e^{-G(x)}\int_0^xe^{G(y)}\theta(y)a_m''(y)dy,\text{ }m> 0.
\end{equation}

Before proceeding to show exponential error for $V$ analytic, we show $O(h^N)$ error for arbitrary $V$.
To complete the proof of $O(h^N)$ quasimodes, we need to estimate
$$\|(P(x, h D) -z)g\|_{L^2}/\|g\|_{L^2}.$$
Let $C$ denote various positive constants that are independent of $h$ and $x$. Then,
\begin{eqnarray}
\|g\|_{L^2}^2&\geq& \int_{-\delta/2}^{\delta/2}|f(x)|^2dx\geq\int_{-\delta/2}^{\delta/2}e^{-6\gamma x^2h^{-1}-C}dx\nn 
&=&\int_{-\delta h^{-1/2}/2}^{\delta h^{-1/2}/2}e^{-6\gamma t^2-C}h^{1/2}dt\geq\int_{-1/2}^{1/2}e^{-6\gamma t^2-C}h^{1/2}dt=Ch^{1/2}.\label{eqn:gLowerBound} 
\end{eqnarray}
Next, we compute
\begin{eqnarray}\|\left(P(x, h D)-z\right)g\|_{L^2}&=&\|h^2f\chi ''+2h^2f'\chi '+hf\chi ' +\chi(P(x, h D) f-zf)\|_{L^2}\nn 
&\leq &h^2\|f\chi ''\|_{L^2}+2h^2\|f'\chi '\|_{L^2} +h\|f\chi '\|_{L^2}+\|h^{N}\chi\phi_{N}e^{i\psi/h}\|_{L^2}.\label{eqn:expressionWithCutoff} 
\end{eqnarray}
Thus, we need to estimate each of the norms. Note that $\chi ' $ and $\chi ''$ have support in $\{x:\delta/2\leq |x|\leq \delta\}$. Thus, we have 
\begin{eqnarray}
\|f\chi ''\|_{L_2}^2&\leq &C_3\int_{\delta/2\leq |x|\leq \delta }e^{-2\gamma x^2h^{-1}+C}dx\leq Ce^{-\gamma \delta^2/2h}.\label{eqn:expErrorAfterCut1}
\end{eqnarray}
Similarly, 
\begin{equation}
\label{eqn:expErrorAfterCut}
\|f'\chi ' \|^2_{L^2}\leq Ce^{-\gamma \delta^2/2h},\quad
\|f\chi '\|^2_{L^2}\leq Ce^{-\gamma \delta^2/2h}.
\end{equation}
Next, observe that 
\begin{eqnarray}
\|h^N\chi\phi_{N}e^{i\psi/h}\|_{L^2}^2&\leq &h^{2N}\|\phi_N\|^2_{L^\infty}\int_{-\delta}^{\delta}e^{-2\gamma x^2h^{-1}+C}dx\leq Ch^{2N}\|\phi_N\|_{L^\infty}^2\int_{-\delta h^{-1/2}}^{\delta h^{-1/2}}e^{-2\gamma x^2+C}h^{1/2}dx\nn 
&\leq &Ch^{2N}\|\phi_N\|_{L^\infty}^2\int_{-\infty}^\infty e^{-2\gamma x^2+C}h^{1/2}dx\leq C\left(h^N\|\phi_N\|_{L^\infty}\right)^2 h^{1/2}.\label{eqn:gUpperBound}
\end{eqnarray} 
Now, $|\phi_m|\leq c_m$ on $\{x:|x|\leq \delta\}$, uniformly for $h\leq\delta^2$. Therefore, 
combining \eqref{eqn:expressionWithCutoff} with inequalities \eqref{eqn:gLowerBound}, \eqref{eqn:expErrorAfterCut1}, \eqref{eqn:expErrorAfterCut}, and \eqref{eqn:gUpperBound}, gives $O(h^N)$ quasimodes for arbitrary $N$. We then normalize to obtain $\varphi$. 

We will now assume that $V(x)$ is real analytic and prove exponential smallness of the error.
\begin{lemma}
\label{l:ana}
Let $ I = [ - \delta, \delta ] $ where $ \delta $ is a small
constant.  Suppose that $ \tau_0$, $\tau_1$, $\tau_2$, and $ d_0 $ are holomorphic
functions of $ z \in \Omega $ and $|\tau_2|\geq\recip{C}$ for some $C>0$ where $ \Omega $ is a neighborhood
of $ I $ in $ \mathbb C$. 

If $ d_m $ is defined inductively by 
\begin{equation}
\label{eq:bmp}  d_{m+1} ( z ) =  \int_0^z \tau_2( \zeta ) d_m'' ( \zeta )+\tau_1(\zeta)d_m'(\zeta)+\tau_0(\zeta)d_m(\zeta)dy .
\end{equation}

Then for some $ C_1 > 0 $, $C_2 >0$ and $[-\delta,\delta]\subset\tilde{\Omega}\subset \Omega$,
\begin{equation}
\label{eq:am1}
\sup_{\tilde{\Omega}} | \partial ^pd_m | \leq  C_2^{p+1}p!C_1^{m+1}  m^m . 
\end{equation}
\end{lemma}
\begin{proof}
Using integration by parts, we obtain that 
$$d_{m+1}(z)=\tau_2(z)d_m'(z)-\tau_2(0)d_m'(0)+(\tau_1(z)-\tau_2 '(z))d_m(z)+\int_0^z (\tau_2 ''(\zeta)-\tau_1'(\zeta)+\tau_0(\zeta))d_m(\zeta)d\zeta.$$
Then, since $\tau_2$ is holomorphic in $\Omega$ and $\inf_\Omega|\tau_2|\geq \recip{C}$ for some $C>0$, we make the conformal change of variables $z\to w$ where 
$$\frac{d w}{dz}=\tau_2(z)^{-1}.$$
Then, letting $b_m=d_m(z(w))$, we have 
$$\partial^p_w(b_{m+1})(w)=\partial^{p+1}_w b_m(w)-\delta_{p0}(\partial_wb_m)(0)+\partial_w^p(\rho_0b_m)(w)+\partial_w^p\int_0^{z(w)}\rho_1(\zeta)b_m(\zeta)d\zeta$$
where $\rho_0=(\tau_1-\tau_2 ')(z(w))$ and $\rho_1=((\tau_2''-\tau_1'+\tau_0)\tau_2)(z(w)).$ Put $\Omega_w:=\{w:z(w)\in \Omega\}$. Then, since the change of variables was conformal, and $\tau_i$, $i=0,1,2$, are holomorphic, we have that there exists $C_\rho>0$ such that 
$$|\partial^p_w\rho_i|_{\Omega_w}\leq C_\rho^{p+1}p^p\quad \text{ for }i=0,1$$
where we define $|f|_\Omega:=\sup_\Omega |f|$ for a function $f$ defined on $\Omega$. 

We claim that for some $C_1>0$, $C_0>C_\rho$,
$$|\partial^p_wb_m|_{\Omega_w}\leq C_0^{p+1}C_1^{m+1}(m+p)^{m+p}.$$ 
We prove the claim by induction. The holomorphy of $b_0$ gives us the base case. We now prove the inductive case. 

By the inductive hypothesis, we have that 
\begin{equation}
\label{eqn:term1b}
|\partial^{p+1}_wb_{m}|_{\Omega_w}\leq C_0^{p+2}C_1^{m+1}(m+p+1)^{m+p+1}.
\end{equation}
Similarly,
\begin{equation}
\label{eqn:term2b}
|b_{m}(0)|_{\Omega_w}\leq C_1^{m+1}m^m.
\end{equation}

Next, we prove similar estimates for $\partial_w^p(\rho_0 b_m)$. 
By Leibniz rule, we have that 
\begin{equation}
\label{eqn:combinationDerivative}
|\partial^p (\rho_0b_m)|_{\Omega_w}=\left|\sum_{k=0}^p\frac{p!}{k!(p-k)!}\partial^{k}\rho_0\partial^{p-k}b_m\right|_{\Omega_w}\leq \sum_{k=0}^pC_0^{p+2}C_1^{m+1}r_{k,m,p}\end{equation}
where
$$r_{k,m,p}:= \frac{p!}{k!(p-k)!}k^k(m+p-k)^{m+p-k}.$$
We claim that for $0\leq k\leq\frac{p}{2}$, $r_{k,m,p}\geq r_{p-k,m,p}$. 
To see this, we write this inequality as
$$k^k(m+p-k)^{m+p-k}\geq (p-k)^{p-k}(m+k)^{m+k},\quad \text{ for }0\leq k\leq \frac{p}{2}.$$
Putting $x:=\frac{k}{m}$ and $y=\frac{p-k}{m}$, the inequality is equivalent to
$$x^x(1+y)^{1+y}\geq y^y(1+x)^{1+x}, \quad \text{for }0\leq x\leq y$$
which follows from the monotonicity of the function $x\mapsto \left(\frac{1+x}{x}\right)^x(1+x).$

Next, observe that, $0\leq k<p-1$,
\begin{eqnarray}
\frac{r_{k+1,m,p}}{r_{k,m,p}}&=&\frac{p-k}{k+1}\frac{(k+1)^{k+1}}{k^k}\frac{(m+p-k-1)^{m+p-k-1}}{(m+p-k)^{m+p-k}}\nn 
&=&\frac{p-k}{m+p-k-1}\left(1+\recip{k}\right)^k\left(1-\recip{m+p-k}\right)^{m+p-k}\nn 
&\leq &\frac{p-k}{m+p-k-1}e^{1-1+\recip{2(m+p-k)}}\leq e^{\recip{2(m+p-k)}},\nonumber 
\end{eqnarray}
where we use $\log (1-x)\leq -x+\frac{x^2}{2}$. Then, since for $0\leq k\leq \frac{p}{2}$, $r_{k,m,p}\geq r_{p-k,m,p}$, we have 
\begin{eqnarray}
|\partial^{p} (\rho_0b_m)|_{\Omega_w}&\leq& 2C_0^{p+2}C_1^{m+1}\sum_{k=0}^{\frac{p}{2}+1}r_{k,m,p} \leq 2C_0^{p+2}C_1^{m+1}\sum_{k=0}^{\frac{p}{2}+1}r_{0,m,p}\prod_{n=0}^{k-1}e^{\recip{2(m+p-n)}}\nn 
&\leq &2C_0^{p+2}C_1^{m+1}\sum_{k=0}^{\frac{p}{2}+1}r_{0,m,p}e^{\frac{k}{2m+p}}\leq 2C_0^{p+2}C_1^{m+1}\sum_{k=0}^{\frac{p}{2}+1}r_{0,m,p}e^{\frac{p+2}{4m+2p}}\nn 
&\leq &C_0^{p+2}C_1^{m+1}(p+2)r_{0,m,p}e^{\frac{1}{2}}\leq e^{\recip{2}}C_0^{p+2}C_1^{m+1}(m+p+1)^{m+p+1}\nonumber 
\end{eqnarray}

Therefore, there exists $M_1>0$ such that 
\begin{equation}
\label{eqn:term3b}
|\partial^p_w(\rho_0b_m)|_{\Omega_w}\leq M_1C_0^{p+2}C_1^{m+1}(m+p+1)^{m+p+1}.
\end{equation}
By analagous argument, there exists $M_2>0$ such that 
\begin{equation}
\label{eqn:term4b}
\left|\partial^p_w\int_0^{z(w)}\rho_1(\zeta)b_m(\zeta)d\zeta\right|_{\Omega_w}=\left|\partial^{p-1}_w((\rho_1\partial_w z)b_m)\right|_{\Omega_w}\leq M_2C_0^{p+1}C_1^{m+1}(m+p)^{m+p}.
\end{equation}

Next, choose $C_1>C_0\left(4\max(M_1,M_2,1)\right).$ Then, combining \eqref{eqn:term1b}, \eqref{eqn:term2b}, \eqref{eqn:term3b}, and \eqref{eqn:term4b}, we have 
$$|\partial^p_wb_{m+1}(w)|_{\Omega_w}\leq C_0^{p+1}C_1^{m+2}(m+p+1)^{m+p+1}.$$
Then, since $w\to z$ is a change of variables independent of $m$ which maps $\Omega_w\to \Omega$ and $b_m(w)=d_m(z(w))$, we have
$$|d_m|_{\Omega}= |b_m|_{\Omega_w}\leq C_0C_1^{m+1}m^m.$$
Now, choose $\tilde{\Omega}\subset \Omega$ with $\inf\{|z-\zeta|:z\in \tilde{\Omega},\text{ }\zeta\in\partial \Omega
\}>\gamma>0$. Then, since $d_m$ are holomorphic, we apply Cauchy estimates to obtain 
$$|\partial^p d_m|_{\tilde{\Omega}}\leq \gamma^{-p}p!|d_m|_\Omega\leq \gamma^{-p}p! C_0C_1^{m+1}m^m.$$
\end{proof}

We apply the lemma with 
$  d_m ( x )   := e^{G ( x ) } a_m ( x ) $, $ \tau_2:=-\theta
$, $\tau_1:=2\theta G'$, and $\tau_0:=\theta(G''-(G')^2)$ where $\theta$ and $G$ are given above.
 Analyticity of $ V $ implies that $ a_0 $, $ \theta $, and $\psi$ are
 holomorphic in a neighbourhood of $ I$.

Then, putting $1/N={eC_1 h}$, using Lemma \ref{l:ana} and that $\psi$ is real analytic, we have 
\begin{eqnarray}
\sup_{x\in [-\delta,\delta]}|h^N\phi_N|&\leq& h^NC(C_2C_1^{N}(N-1)^{N-1}+2C_2^2C_1^{N}(N-1)^{N-1}+6C_2^3C_1^N(N-2)^{N-2})\nn 
&\leq& C h^NC_1^{N}N^{N}\leq C (h^NC_1^{N}\left(\recip{eC_1h}\right)^N)= C e^{-N}= Ce^{-\recip{eC_1h}}\label{eqn:expSmall}
\end{eqnarray}
where $C$ denotes various postive constants that are independent of $N$.
Finally, combining \eqref{eqn:expressionWithCutoff} with inequalities \eqref{eqn:gLowerBound}, \eqref{eqn:expErrorAfterCut1}, \eqref{eqn:expErrorAfterCut}, \eqref{eqn:gUpperBound}, and \eqref{eqn:expSmall} gives $O(e^{-\recip{Ch}})$ quasimodes. We then normalize $g$ to obtain $\varphi$. 
\end{proof}

Now, applying this result to $P_1(x, hD)$, we obtain
$$\psi=\int_0^x\frac{i}{2}+\left(\xi_0-\frac{i}{2}\right)\left(1-\frac{2x_0t+t^2}{\left(\xi_0-\frac{i}{2}\right)^2}\right)^{1/2}dt$$
for $x_0<0$ and $z=-\xi_0^2+i\xi_0+\mu-x_0^2$.

\begin{figure}[htbp]
\begin{center}
\includegraphics[width=6in]{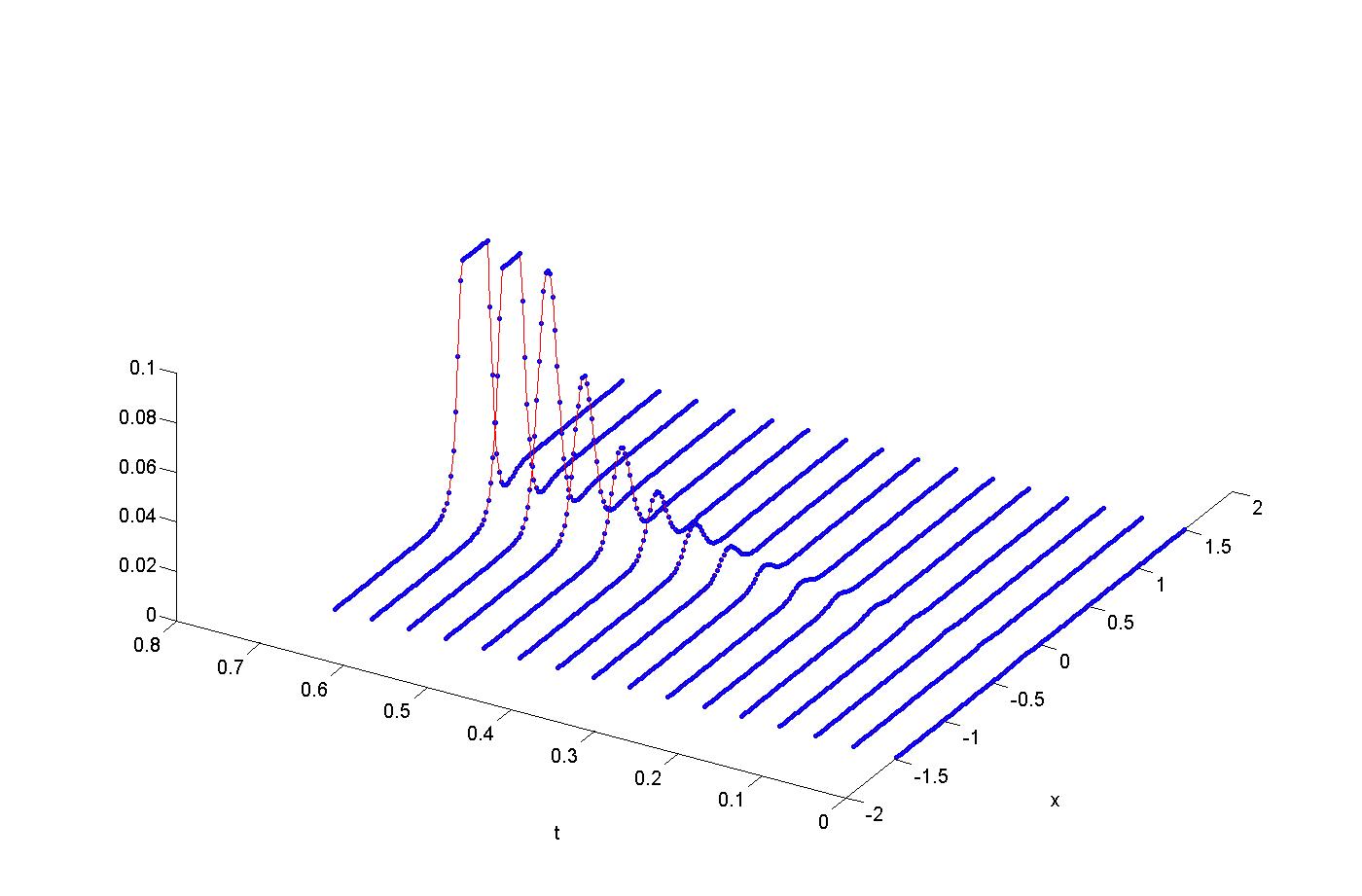}
\caption{\label{f:3}
We set $h=10^{-2}$ and see that the difference between the solution to (\ref{eqn:main}) with initial data a quasimode with error $O(h^2)$ (red line) and the solution with initial data a quasimode with error $O(h^3)$ (blue dots) is negligible. Thus, by using $O(h^3)$ error quasimodes, we have not introduced large error into our numerical calculations.}
\end{center}
\end{figure}

\section{Instability}
\label{sec:i}

\begin{figure}[htbp]
\begin{center}
\includegraphics[width=6in]{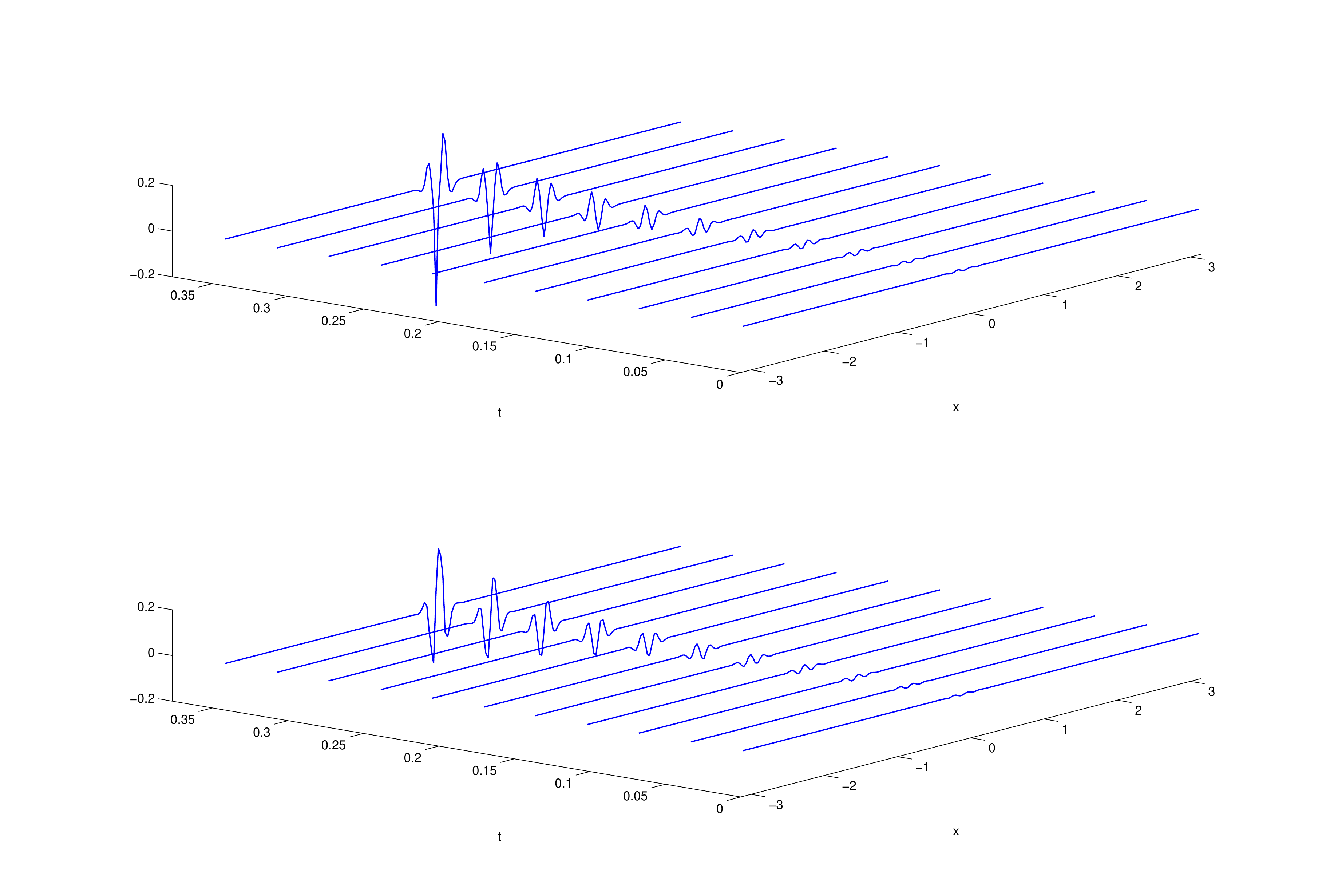}
\caption{\label{f:4}
We show a numerical simulation $u(t)$ of the evolution of \eqref{eqn:main} with a quasimode at imaginary energy as initial data. Specifically, we set $h=1/193$ and use a quasimode with eigenvalue $z=\recip{16}+\frac{i}{4\sqrt{2}}$. The real part is shown in the top graph and the imaginary part in the bottom graph. We see that although subsolution methods do not apply to these quasimondes, blow up still occurs in time $\approx 0.35$.}
\end{center}
\end{figure}

Our approach to obtaining blow-up of (\ref{eqn:main}) will follow that used by Sandstede and Scheel in \cite{Sand03}. We will first demonstrate that, from small initial data, we obtain a solution that is $\geq 1$ on a deformed ball in time $t_1=O(1)$. We will then use the fact that the solution is $\geq 1$ on this region to demonstrate that after an additional $t_2=O(h)$, the solution to the equation blows up.

First, we prove that there exists initial data so that the solution to (\ref{eqn:main}) is $\geq 1$ in time $O(1)$. Recall that $\varphi_t$ denotes the flow of $i\la\nabla \rho,D\ra$. 
\begin{lemma}
\label{lem:O(1)ndim}
Fix $\mu>0$, $\alpha<\mu$, $0<\e\leq\recip{2}(\mu-\alpha)$, and $(x_0,a,\delta)\in \re^d\times \re\times \re^+$ such that both $\varphi_t(B\left(x_0,2a\right))\subset V^{-1}[0,\mu-\alpha -\e]$ for $t\leq \delta$ and $\varphi_t$ is defined on $B\left(x_0,2a\right)$ for $0\leq t<2\delta$. Then, for each 
$$0<h<h_0$$
 where $h_0$ is small enough, there exists
$$u_0(x)\geq 0,\quad \|u_0\|_{C^P}\leq \exp\left(-\recip{Ch}\right),\quad p=0,1,...$$ 
and $0<t_1<\delta$ so that the solution to (\ref{eqn:main}) with initial data $u_0$ satisfies $u(x,t_1)\geq 1$ on $x\in \varphi_{t_1}(B(x_0,a))$.
\end{lemma}

\noindent\begin{proof}
Let $\upsilon$ solve 
\begin{equation}
\label{eqn:linProb}
(h\partial_t-P(x,hD))\upsilon=0, \upsilon(x,0)=\upsilon_0.
\end{equation}

Let $w_0:\re^d\to \re$ and define $O:=\{x:w_0>0\}$. We make the following assumptions on $w_0$,
\begin{equation}\label{eqn:wAssume1}
w_0\geq 0,\quad\|w_0\|_{C^p}\leq \exp(-\recip{Ch}),\quad w_0\in C(\re^d)
\end{equation}
\begin{equation}\label{eqn:wAssume2}
w_0\in \overline{C^\infty}(O),\quad \text{supp }w_0\subset B(x_0,2a),\quad
w_0>\exp\left(-\frac{\delta}{2h}\right)\text{ on }B(x_0,a),
\end{equation}
\begin{equation}\label{eqn:wAssume3}
\partial O\text{ is smooth},\quad -\Delta w_0 (x)\leq Cw_0(x)-\beta\text{ for }x\in O\text{ and }0<h<h_0.
\end{equation}
where $\overline{C^\infty}(O)$ are smoothly extendible functions on $O$.
We will construct such a function at the end of the proof.

Define $w:[0,2\delta)\times \re^d\to \re$ by 
 $$w:=\begin{cases}
	\exp\left(\frac\alpha h t\right)w_0(\varphi_t(x))&\text{ where }\varphi_t\text{ is defined}\\
0&\text{else}.\end{cases}$$
Since supp $w\subset B(0,2a)$ and $\varphi_t$ is defined on $B(0,2a)\times [0,2\delta)$, $w$ is continuous. We proceed by showing that $w$ is a viscosity subsolution of $(\ref{eqn:linProb})$ in the sense of Crandall, Ishi, and Lions \cite{Lions}. 

First, we show that $w$ is a subsolution on $O_t:=\varphi_t(O)$ for $t<\delta$. 
\begin{eqnarray}
hw_t-P(x,hD)w&=&hw_t-h^2\Delta w-ih\la \nabla \rho ,D\ra w-\mu w +V(x)w\nn 
&=&(\alpha -\mu +V(x))w-h^2\Delta w\nn 
&\leq & \exp\left(\frac \alpha ht\right)
\left((\alpha - \mu +V(x))w_0\right)-h^2\Delta w\nonumber
\end{eqnarray}
Now, by Taylor's formula, $\varphi_t(x)=x+O(t)$. Hence $-\Delta\left[w_0(\varphi_t(x))\right]=-\Delta w_0(x)+O(t)$.  We have $t<\delta$, and $-\Delta w_0\leq Cw_0-\beta$ on $O$. Therefore, for $\delta$ small enough, $-\Delta w\leq Cw_0 $. Hence, for $h$ small enough independent of $0<\delta <\delta_0$,
$$hw_t-P(x,hD)w\leq \exp\left(\frac{\alpha}{h}t\right)\left( \alpha-\mu+Ch^2+V(x)\right)w_0$$
Now, since for some $\e>0$ and $t<\delta$, $\text{supp } w\subset V^{-1}[0,\mu-\alpha -\e]$ we have that 
$$hw_t-P(x,hD)w\leq \exp\left(\frac{\alpha}{h}t\right)(Ch^2-\e)w\leq 0$$
for $h$ small enough. Thus, $w$ is a subsolution on $O_t$ for $t<\delta$. Next, we observe that on $(\re^d\setminus \overline{O_t})$, $w\equiv 0$ and hence is a subsolution of (\ref{eqn:linProb}).  

Finally, we consider $\partial O_t:= \varphi_t(\partial O)$. We have that $\partial O_t$ is smooth. If $y_0\in \partial O_t$ and $w$ is twice differentiable at $y_0$, then $w_t=(\Delta w)(y_0)=(D w)(y_0)=w(y_0)=0$ and $w$ is clearly a subsolution to \eqref{eqn:linProb} at $y_0$. Let $y_0\in \partial O_t$ be a point where $w$ is not twice differentiable. Suppose that $\phi\in C^2$ such that $w-\phi$ has a maximum at $y_0$. 

We take paths through $y_0$ to reduce to a one dimensional problem. For any path $\gamma:I\to \re\times \re^d$ with $\gamma(0)=(t,y_0)$, define $h_\gamma(s):=w(\gamma(s))$ and $\phi_\gamma(s):=\phi(\gamma(s))$. Since $w$ is nonnegative, continuous, smooth on $O_t$, and extends smoothly from $O_t$ to a function on $\re^{d}$ for all $t<\delta$, $h'_{\gamma+}:=\lim_{s\to 0^+}h_\gamma '(s)$ and $h'_{\gamma-}:=\lim_{s\to 0^-} h_\gamma '(s)$ exist. Simlarly, $h''_{\gamma+}$ and $h''_{\gamma-}$ exist. Therefore, since $w-\phi$ is maximized at $y_0$, $h'_{\gamma+}\leq \phi_\gamma '(0)\leq h'_{\gamma-}$.  Now, since $w$ is not twice differentiable at $y_0$, either there exists $\gamma$ such that $h_\gamma$ is not differentiable at $0$ or $w$ is differentiable at $y_0$, but not twice differentiable.
\begin{description}
\item[Case 1] $\gamma(s)$ is a path through $y_0$ for which $h_\gamma$ is not differentiable.\\
Then $h'_-<h'_+$ and there exists no such $\phi$. 
\item[Case 2] $w$ is differentiable at $x_0$ but not twice differentiable. \\
Then, for all $\gamma$ through $y_0$, $\varphi_\gamma '(0)=h'_\gamma(0)$ and $\varphi_\gamma ''(0)\geq \max(h''_{\gamma +},h''_{\gamma-})$. Now, let $\gamma_i$ be the coordinate paths through $x_0$ with $w(\gamma_i(t))>0$ for $0<t<\delta$. Then, since on $w>0$, $w$ is a subsolution of the linearized problem (\ref{eqn:linProb}), we have 
\begin{eqnarray} 
(h\partial_t\phi-P(x,hD)\phi)(x_0)&=&h\partial_t w-h^2\Delta \phi-ih\la \nabla \rho, D\ra w-\mu w+V(x)w\nn 
&\leq& h\partial_t w-h^2\left(\sum_i\lim_{t\to 0^{+}}h''_{\gamma_i+}\right)-ih\la \nabla \rho, D\ra w-\mu w+V(x)w\leq 0.\nonumber
\end{eqnarray}
\end{description} 
Thus, we have that $w$ is a subsolution on $\partial O_t$. Putting this together with the arguments above, we have that $w$ is a viscosity subsolution for (\ref{eqn:linProb}) on $t<\delta$. 

Now, by an adaptation of the maximum principle found in \cite[Section 3]{Lions} to parabolic equations, any solution, $\upsilon$ to (\ref{eqn:linProb}) with initial data $\upsilon_0>w_0$ will have $\upsilon\geq w$ for $t<\delta$. But, since $\upsilon\geq 0$, $\upsilon^3\geq 0$ and hence the solution $u$ to (\ref{eqn:main}) with initial data $\upsilon_0$ will have $u\geq \upsilon\geq w$ for $t<\delta$. Now, since for $t>\frac{\delta}{2}$, $w(x,t)\geq 1$ on $\varphi_t(B(x_0,a))$, we have the result.

We now construct a function $v_0$ satsifying the assumptions, (\ref{eqn:wAssume1}),\eqref{eqn:wAssume2}, and \eqref{eqn:wAssume3}. Let $v_1$ be the ground state solution of the Dirichlet Laplacian on $B(0,1)\subset \re^d$ i.e.
$$-\Delta v_1=\lambda v_1\text{ on }B(0,1)\quad v_1|_{\partial_{B(0,1)}}=0.$$
Then, $v_1$ extends smoothly off of $B(0,1)$ and has $v_1> 0$ inside $B(0,1)$ (see for instance \cite[Section 6.5]{E}).  

Let $\chi\in C^\infty(B(0,1))$, $0\leq \chi \leq 1$ with $\chi\equiv 1$ on $B(0,1)\setminus B(0,1-\e)$ and supp $\chi \subset B(0,1)\setminus B(0,1-2\e)$. Then, define $v_2:= Mv_1+\left[\chi(x)\right] (|x|^2-1)$ where $M$ is large enough so that $v_2> 0$. There exists such an $M$ since $v_1>0$ in $B(0,1)$ and $\lambda >0$ imply that $-\Delta v_1=\lambda v_1> 0$ and hence, by Hopf's Lemma, $\partial_\nu v_1<0$ on $\partial B(0,1)$ where $\nu$ is the outward normal vector to $\partial B(0,1)$ (see for instance \cite[Section 6.4]{E}).

For $|x|\leq 1-\e$, there exists $C>0$ such that,
\begin{equation}
-\Delta v_2=\lambda Mv_1-\left[\Delta \chi(x)\right](|x|^2-1)-4\la \partial \chi, x\ra -2\chi d\leq \lambda Mv_1+C.\nonumber
\end{equation}
For $|x|\leq 1-\e$, $v_1>\delta$. Thus, by increasing $M$ if necessary, we obtain $\beta>0$ such that
$$-\Delta v_2\leq \lambda M v_1+C\leq   2\lambda v_2-\beta.$$

Now, for $1-\e<|x|<1$, 
$$-\Delta v_2=\lambda Mv_1-2d=\lambda (v_2-|x|^2+1)-2d\leq \lambda v_2+\lambda (2\e -\e^2)-2d.$$
Thus, for $\e>0$ small enough, there exists $
\beta>0$ such that 
$$-\Delta v_2\leq \lambda v_2-\beta.$$

 Finally, $\exists$ $a\in \re$, $x_0\in \re^d$, and $C_1,C_2>0$ constants so that 
$$v_0=\begin{cases}
C_1e^{-\recip{C_2h}}v_2(a^{-1}(x-x_0))&x\in B(x_0,a)\\
0&\text{else}
\end{cases}$$ satisfies the conditions on $w_0$. 

\end{proof}

\noindent{\bf Remark 1.} If a shorter time is desired, one may use initial data of $O(h^n)$ to obtain the same result in time $O(h|\log h|)$.

\noindent{\bf Remark 2.} Notice that to obtain a growing subsolution it was critical that $\mu>0$. This corresponds precisely with the movement of the pseudospectrum of $P(x, hD)$ into the right half plane.

Now, we will demonstrate finite time blow-up using the fact that in time $O(1)$ the solution to (\ref{eqn:main}) is $\geq 1$ on an open region. The proof of theorem \ref{thm:main} follows

\begin{figure}[htbp]
\begin{center}
\includegraphics[width=6in]{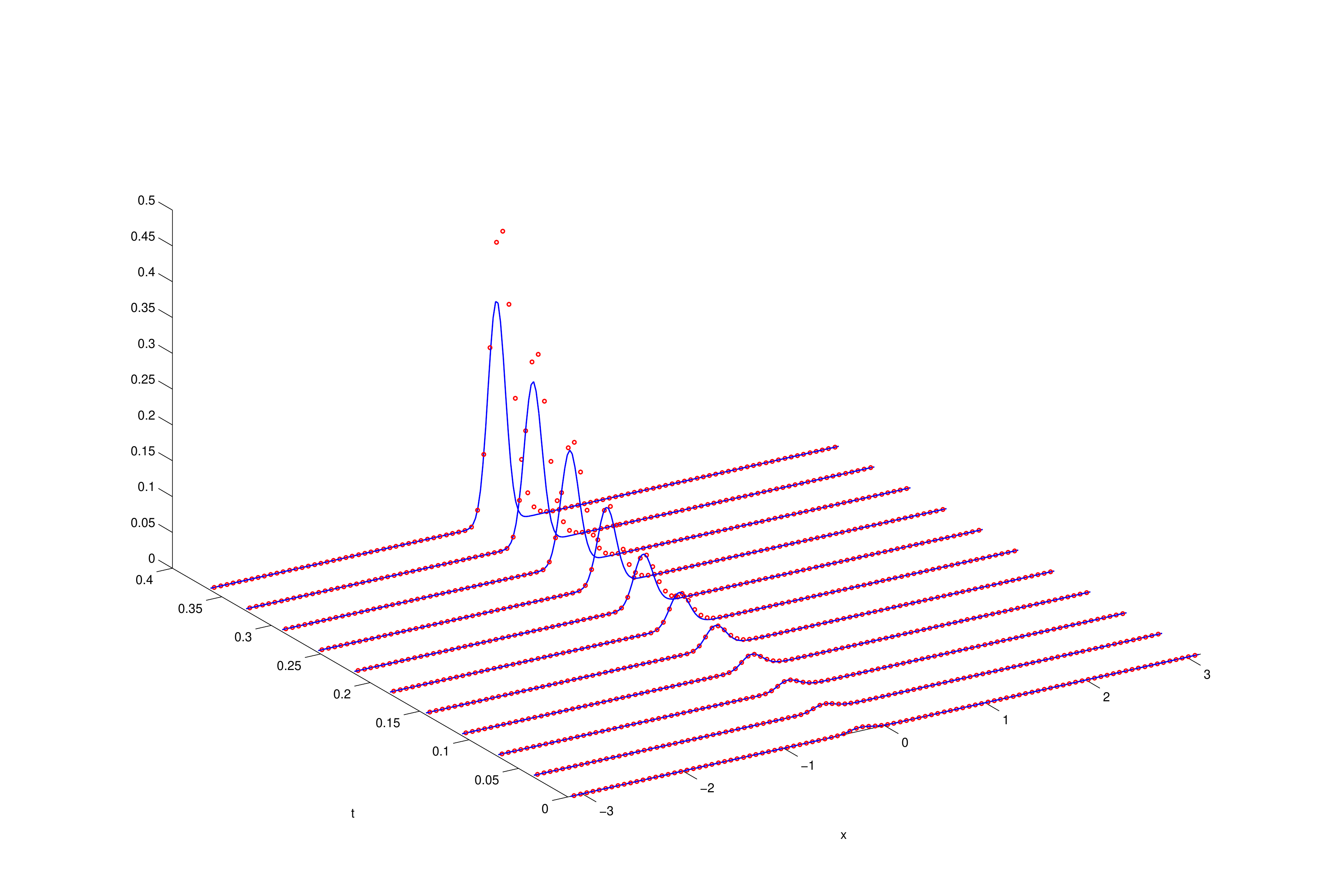}
\caption{\label{f:1}
We show simulations of solutions to the equation $hu_t=P_1(x,hD)u$ with $h=1/193$. The solution using a quasimode $u_0$ with eigenvalue $z=\recip{16}$ and error $O(h^3)$ as initial data is shown in the blue solid linse. The red dotted lines show $u(x,t)=u_0(x)\exp\left(zt/h\right)$. We see that the solution to the linearized problem \eqref{eqn:linProb} with quasimode initial data closesly approximates the exponential until $t\approx0.3$}
\end{center}
\end{figure}

\begin{proof}
Let $u_0(x)$ and $t_1$ be the initial data and time found in Lemma \ref{lem:O(1)ndim} with $(a, x_0,\delta)$ such that $\varphi_t$ is defined  on $B(x_0,a)$, $\varphi_{t}\left(B\left(x_0,a\right)\right)\subset V^{-1}[0,\frac{\mu}{2}]$ for $t\in [0,\delta]$, and $t_1<\delta$. Then, $u(x,t_1)\geq 1$ on $\varphi_t\left(B\left(x_0,a\right)\right)$. 

Now, let $\Phi\in C_0^\infty(\re)$ be a smooth bump function with $\Phi(y)=1$ on $|y|\leq 1$, $0\leq \Phi \leq 1$, supp $\Phi\subset (-2,2)$, and $\Phi ''\leq C\Phi^{1/3}$ (one such function is given by $e^{-1/x}$ for $\e>x>0$). Define $\chi:\re^d\to \re$ by $\chi(y):=\Phi\left(2a^{-1}|y|\right).$ 

Next, let $y'=\varphi_{t}(x_0+y)$ and let
$$v(y,t):=\chi(y)u(y',t).$$
Then, we have that 
$$hv_t=h^2\Delta v+\mu v+v^3-2h^2\la\nabla\chi ,\nabla u\ra -h^2\Delta \chi u+(\chi -\chi ^3)u^3-V(y')v.$$
Finally, define the operations, $[f]$ and $[ f,g]$ by
$$[f]:=\negint_{B(0,a)}f(y)dy\quad \text{ and }\quad [ f,g ]:=\negint_{B(0,a)}\la f(y),g(y)\ra dy.$$

Then, we have that 
\begin{eqnarray}
h[v]_t&=&h^2[\Delta v]+\mu [v]+[v^3]-h^2\left[\Delta \chi ,u\right] -2h^2\left[ \nabla\chi ,\nabla u\right] +\left[ \chi -\chi ^3,u^3\right] -\left[ V(y'),v\right] \nn 
&\geq &\mu [v]+[v^3]+h^2\left[ \Delta \chi ,u\right] +\left[ \chi -\chi ^3,u^3\right]-\left[ V(y'),v\right] \label{eqn:intPartsnd}
\end{eqnarray}
Here, (\ref{eqn:intPartsnd}) follows from integration by parts, the fact that $\nabla \chi=0$ at $|y|= a$ and that 
$$\negint_{B(0,a)} \Delta v=\frac{c_d}{a^d}\int_{\partial B(0,a)}\nabla v\cdot \nu=0.$$

\noindent We will later need that $[v^3]\geq [v]^3$. To see this use H\"{o}lder's inequality with $f=wc_d^{\recip{3}}(a)^{-\frac{d}{3}}$, $g=c_d^{\frac{2}{3}}(a)^{-\frac{2d}{3}}$, $p=3$, and $q=\frac{3}{2}$ to obtain
$$[v^3]=\int_{B(0,a)}\frac{v^3c_d}{a^d}d\xi\geq \left(\int_{B(0,a)}\frac{vc_d}{a^d}d\xi \right)^3=[v]^3.$$

Next, we will estimate $\left[ \Delta\chi , u\right]$.  Observe that 
\begin{eqnarray}\left[ \Delta\chi ,u\right] =\left|\negint\Delta \chi  u\right|&\leq &\left|\recip{C_d}\int_0^a\left(\frac{2a^{-1}(d-1)}{r}\Phi '(2a^{-1}r) +4a^{-2}\Phi ''(2a^{-1}r)\right) \int_{\partial B(0,r)} u(r,\phi) dSdr\right| \nn 
&\leq &C\int_0^a\left(\Phi^{1/2}+\Phi^{1/3}\right)\int_{\partial B(0,a)}u(r,\phi)dSdr\label{eqn:chigreat0}\\
&\leq& C\negint\chi ^{1/3}u\leq C\negint(1+\chi u^3)\leq C'+C\negint\chi u^3\label{eqn:cubicnd}
\end{eqnarray}
where $C'$ and $C$ do not depend on $h$. Here (\ref{eqn:chigreat0}) follows from the fact that for any function $\Phi\geq 0$, $\Phi '\leq \Phi^{1/2}$ and that $\Phi '=0$ near $r=0$. (\ref{eqn:cubicnd}) follows from $0\leq \Phi\leq 1$.

Now, we have 
\begin{eqnarray}
h[v]_t&=&\mu [v]+[v^3]+h^2\left[ \Delta \chi ,u\right] +\left[ \chi -\chi ^3,u^3\right]-\left[ V(y'),v\right]\nn 
&\geq &\mu [v]+[v^3]-O(h^2)+\left[ (1- O(h^2))\chi -\chi^3,u^3\right] -\left[ V(y'),v\right]\nn 
&\geq &\mu [v]+[v^3]-O(h^2)-O(h^2)[v^3] -\left[ V(y'),v\right]\label{eqn:chiCubeSubnd}\\ 
&\geq&\mu [v] +(1-O(h^2))[v]^3-O(h^2)-\left[ V(y'),v\right]\label{eqn:holder}
\end{eqnarray}
Here, (\ref{eqn:chiCubeSubnd}) follows from the fact that $\chi \leq 1$ and (\ref{eqn:holder}) follows for $h<1$ since  $[v^3]\geq [v]^3$.

Now, on $t< \delta$, we have $V(y')\leq \frac{\mu}{2}$. Thus, for $0<t<\delta$,
$$h[v]_t\geq \frac{\mu}{2}[v] +(1-O(h^2))[v]^3-O(h^2).$$
We have that $[v](t_1)\geq 1/4$ and $\mu>0$. Therefore there exists $\gamma>0$ such that, for $h$ small enough and $t_1\leq t\leq t_1+\gamma$,
$$h[v]_t\geq \frac{\mu}{4}[v]+\recip{2}[v]^3.$$ 
But, the solution to this equation with initial data $[v](0)\geq 1/4$ blows up in time $t_2=O(h)$. Hence, so long as $t_1+t_2<\min(\delta,t_1+\gamma)$ and $h$ is small enough, $[v]$ blows up in time $t_1+t_2$. Observe that since $t_1<\delta$, $0\leq t_1+t_2=t_1+O(h)<\min(\delta,t_1+\gamma)$ for $h$ small enough. Thus, the solution to \ref{eqn:main} blows up in time $O(1)$. 
\end{proof}
{\bf{Remark.}} A similar result holds for polynomially small data with blow up in time $O(h|\log h|)$. 

\section{Numerical Simulations}\label{sec:numerical}

We expect that the instability of (\ref{eqn:main}) is related to the presence of pseudospectrum in the right half plane. In fact, using numerical simulations for (\ref{eqn:linProb}) based on code from \cite{KassamTrefethen} with $P$ as in (\ref{eqn:L1d}), (see Figure \ref{f:4}) we are able to demonstrate that the the solution with a quasimode for a positive eigenvalue as initial data closely approximates an exponential. Based on these results we expect that a proof of blow-up using quasimodes will allow the results of Theorem \ref{thm:main} to be extended to complex energies and wider classes of operators. 

All simulations were run in 1 dimension with $\mu=\recip{8}$. Unless otherwise stated, all quasimodes are constructed with errors of $O(h^3)$.



\begin{thebibliography}{10}


\bibitem{CC}
C. Cossu and J. M. Chomaz, 
\newblock
Global measures of local convective instabilities.
\newblock
Phys. Rev. Lett. 78, 4387–4390 (1997)

\bibitem{Lions}
M. Crandall , G. Ishii, P.-L. Lions, 
\newblock 
User's Guide to Viscosity Solutions of Second Order Partial Differential Equations
\newblock 
Bulletin of the American Mathematical Society, 27:1--67, July 1992.

\bibitem{Davi98}
E.B. Davies
\newblock
Semi-classical states for non-self-adjoint schrodinger operators. 
\newblock 
Communications in Mathematical Physics, 200:35–41, 1999

\bibitem{De} N. Dencker, 
\newblock
The pseudospectrum of systems of semiclassical operators,
\newblock
Analysis \& PDE 1:323--373, 2008.

\bibitem{zworski04}
N. Dencker, J. Sj\"ostrand,  and M. Zworski, 
\newblock Pseudospectra of semi-classical (pseudo)differential
operators. 
\newblock
Comm. Pure Appl. Math, 57(3), 2004.

\bibitem{ET}
L.N.~Trefethen and M. Embree, 
\newblock 
{\em Spectra and Pseudospectra: The Behavior of Nonnormal Matrices and Operators},
\newblock
Princeton University Press, 2005.

\bibitem{E}
L.C. Evans, 
\newblock{\em Partial Differential Equations, 2nd edition.}
\newblock 
Graduate Studies in Mathematics, AMS, 2010. 

\bibitem{EZB}
L.C. Evans and M. Zworski.
\newblock \emph{Semiclassical Analysis,} 
\newblock
Graduate Studies in Mathematics, AMS, to appear. 

\bibitem{KassamTrefethen}
A. Kassam and L. N. Trefethen,
\newblock \emph{Fourth-order time stepping for stiff pdes},
Siam J. Sci. Comput. 26 (2005), 1214-1233.

\bibitem{lee}
J. M. Lee, 
\newblock{\em Introduction to Smooth Manifolds},
Springer-Verlag, 2003.


\bibitem{KPS}
K. Pravda-Starov,
\newblock
Pseudo-spectrum for a class of semi-classical operators, 
\newblock
Bulletin de la Société Mathématique de France, 136:329--372, 2008.

\bibitem{Sand03}
B. Sandstede and A. Scheel, 
\newblock
Basin boundaries and bifurcations near convective instabilities: a case study. 
\newblock
J. Differential Equations, 208:176--193, 2005.


\bibitem{zworski01}
M. Zworski
\newblock
{ A remark on a paper by E.B. Davies.}
\newblock
Proc. A.M.S. 129: 2955--2957, 2001.

\end{thebibliography}

\end{document}